\documentclass[a4paper,12pt,reqno]{amsart}
\usepackage{a4wide}
\usepackage{amsmath}
\usepackage{amssymb}
\usepackage{amsthm}
\usepackage{latexsym}
\usepackage{graphicx}
\usepackage[english]{babel}
              
\newtheorem{prop}[subsection]{Proposition}
\newtheorem{conj}[subsection]{Conjecture}
\newtheorem{teor}[subsection]{Theorem}
\newtheorem{lema}[subsection]{Lemma}
\newtheorem{cor} [subsection]{Corollary}
\theoremstyle{definition}
\newtheorem{dfn} [subsection]{Definition}
\theoremstyle{remark}
\newtheorem{obs} [subsection]{Remark}

\newcommand{\Zng}{$\mathbb Z^n$-graded $S$-module}

\newcommand{\me}{\mathfrak m}
\newcommand{\oS}{\overline S}
\newcommand{\St}{\mathcal S}
\def\sdepth{\operatorname{sdepth}}
\def\hdepth{\operatorname{hdepth}}
\def\depth{\operatorname{depth}}
\def\supp{\operatorname{supp}}
\def\deg{\operatorname{deg}}

\def\Ass{\operatorname{Ass}}

\numberwithin{equation}{section}

\begin{document}

\title[On the sdepth and hdepth of some classes of edge ideals of graphs]{On the Stanley depth and Hilbert depth of some classes of edge ideals of graphs}
\author[Andreea I.\ Bordianu,  Mircea Cimpoea\c s]
  {Andreea I.\ Bordianu$^1$ and Mircea Cimpoea\c s$^2$}
\date{}

\keywords{Stanley depth, Hilbert depth, Edge ideal, Path graph, Cycle graph}

\subjclass[2020]{05A18, 06A07, 13C15, 13P10, 13F20}

\footnotetext[1]{ \emph{Andreea I.\ Bordianu}, National University of Science and Technology Politehnica Bucharest, Faculty of
Applied Sciences, 
Bucharest, 060042, E-mail: andreea.bordianu@stud.fsa.upb.ro}
\footnotetext[2]{ \emph{Mircea Cimpoea\c s}, National University of Science and Technology Politehnica Bucharest, Faculty of
Applied Sciences, 
Bucharest, 060042, Romania and Simion Stoilow Institute of Mathematics, Research unit 5, P.O.Box 1-764,
Bucharest 014700, Romania, E-mail: mircea.cimpoeas@upb.ro,\;mircea.cimpoeas@imar.ro}

\begin{abstract}
We study the Stanley depth and the Hilbert depth of the edge ideals of path graphs, cycle graphs, 
generalized star graphs and double broom graphs.

\end{abstract}

\maketitle

\section{Introduction}

Let $K$ be a field and $S=K[x_1,\ldots,x_n]$ the polynomial ring over $K$.
Let $M$ be a \Zng. A \emph{Stanley decomposition} of $M$ is a direct sum $\mathcal D: M = \bigoplus_{i=1}^rm_i K[Z_i]$ as a 
$\mathbb Z^n$-graded $K$-vector space, where $m_i\in M$ is homogeneous with respect to $\mathbb Z^n$-grading, 
$Z_i\subset\{x_1,\ldots,x_n\}$ such that $m_i K[Z_i] = \{um_i:\; u\in K[Z_i] \}\subset M$ is a free $K[Z_i]$-submodule of $M$. 
We define $\sdepth(\mathcal D)=\min_{i=1,\ldots,r} |Z_i|$ and $$\sdepth(M)=\max\{\sdepth(\mathcal D)|\;\mathcal D\text{ is 
a Stanley decomposition of }M\}.$$ The number $\sdepth(M)$ is called the \emph{Stanley depth} of $M$. 

Herzog, Vladoiu and Zheng show in \cite{hvz} that $\sdepth(M)$ can be computed in a finite number of steps if $M=I/J$, 
where $J\subset I\subset S$ are monomial ideals. 
In \cite{apel}, J.\ Apel restated a conjecture firstly given by Stanley in 
\cite{stan}, namely that $$\sdepth(M)\geq\depth(M),$$ for any \Zng $\;M$. This conjecture proves to be false, in general, for 
$M=S/I$ and $M=J/I$, where $0\neq I\subset J\subset S$ are monomial ideals, see \cite{duval}, but remains open for $M=I$.


In \cite{uli}, Uliczka introduced a new invariant associated to a finitely generated graded $S$-module $M$, called \emph{Hilbert depth}, and
denoted by $\hdepth(M)$, which gives a natural upper bound for $\sdepth(M)$. The Hilbert depth of $M$ is the maximal depth of a finitely
graded $S$-module $N$ with the same Hilbert series as $M$. Another equivalent characterization of the Hilbert depth is
$$\hdepth(M)=\max\{r\;:\;(1-t)^rH_M(t)\text{ is positive}\}.$$
For further details regarding the Stanley depth and Hilbert depth we refer the reader to \cite{bruns} and \cite{her}.

Using the last aforementioned characterization for Hilbert depth, in \cite[Theorem 2.4]{lucrare2} we proved that if 
$0\subset I\subsetneq J\subset S$ are two squarefree monomial ideals, then:
$$\hdepth(J/I)=\max\{q\;:\;\beta_k^q(J/I)=\sum_{j=0}^k (-1)^{k-j}\binom{q-j}{k-j}\alpha_j(J/I)\geq 0\text{ for all }0\leq k\leq q\},$$
where $\alpha_j(J/I)$ is the number of squarefree monomials of degree $j$ in $J\setminus I$, for $0\leq j\leq n$. 
Our aim is to use this combinatorial characterization, in order to study the Hilbert depth of edge ideals for several classes of graphs.

Let $I_n=(x_1x_2,x_2x_3,\ldots,x_{n-1}x_n)\subset S$ be the edge ideal of a path graph of length $n-1$.
In Theorem \ref{cory} we give combinatorial formulas for $\hdepth(I_n)$ and $\hdepth(S/I_n)$. Moreover, we strongly believe that
$$\hdepth(I_n)\geq \left\lfloor \frac{2n+1}{3} \right\rfloor.$$
See Conjecture \ref{conj} and Remark \ref{obsy}.

Let $J_n=I_n+(x_nx_1)\subset S$ be the edge ideal of cycle graph of length $n$. 
In Theorem \ref{cory2} we give combinatorial formulas for $\hdepth(J_n)$ and $\hdepth(S/J_n)$.
Also, we conjecture that \small
$$\hdepth(I_n)-\hdepth(J_n),\;\hdepth(S/I_n)-\hdepth(S/J_n) \in \{0,1\}\text{ and }\hdepth(J_n)\geq \left\lfloor \frac{2n+1}{3} \right\rfloor.$$
\normalsize See Conjecture \ref{conj2}. Moreover, we believe that $\hdepth(I_n)=\hdepth(J_n)$ with a probability of $\frac{2}{3}$, while
$\hdepth(S/I_n)-\hdepth(S/J_n)$ with a probability of $\frac{5}{6}$; see Conjecture \ref{conj3}.
In Theorem \ref{cory3} we show that 
$$\hdepth(J_n/I_n)=\hdepth(K[x_1,\ldots,x_{n-4}]/I_n)+2\text{ for all }n\geq 6.$$
Note that similar formulas hold for $\depth$ and $\sdepth$; see \cite[Proposition 1.10]{mir2}.

In the last section we present several applications of our results. Let $I\subset S$ be the edge ideal of a $(k,n_1,\ldots,n_k)$ generalized star graph; 
see Definition \ref{gsg} for details. In Theorem \ref{t4} we give an upper and a lower bound for $\sdepth(S/I)$. In particular, we note that 
$$\hdepth(S/I)\geq \sdepth(S/I) \geq \left\lceil \frac{n_1+\cdots+n_k+k}{2} \right \rceil.$$
Let $I\subset S$ be the edge ideal of the double broom graph $P(n_1,n,n_2)$; see Definition \ref{dbg}. In Theorem \ref{teor2} we prove that
$$ \hdepth(S/I)\geq \sdepth(S/I)=\depth(S/I)=2+  \left\lceil \frac{n-2}{3} \right \rceil.$$
Also, in Proposition \ref{p10} we prove that 
$$ \hdepth(I)\geq \sdepth(I) \geq \left\lceil \frac{n_1+n_2+n+1}{2} \right \rceil.$$
Finally, in Proposition \ref{p12} we compute $\beta_k^q(I)$ and $\beta_k^q(S/I)$ in the case $n=2$.

\pagebreak

\section{Preliminaries}

First, we recall the well known Depth Lemma, see for instance \cite[Lemma 2.3.9]{real}. 

\begin{lema}\label{lem1}(Depth Lemma)
If $0 \rightarrow U \rightarrow M \rightarrow N \rightarrow 0$ is a short exact sequence of modules over a local ring $S$, 
or a Noetherian graded ring with $S_0$ local, then
\begin{enumerate}
\item[(1)] $\depth M \geq \min\{\depth N,\depth U\}$.
\item[(2)] $\depth U \geq \min\{\depth M,\depth N +1\}$.
\item[(3)] $\depth N \geq \min\{\depth U-1,\depth M\}$.
\end{enumerate}
\end{lema}




In \cite{asia}, A.\ Rauf proved the following analog of Depth Lemma:

\begin{lema}\label{asia}(Sdepth Lemma)
If $0 \rightarrow U \rightarrow M \rightarrow N \rightarrow 0$ is a short exact sequence of finitely generated $\mathbb Z^n$-graded $S$-modules, then
$$\sdepth(M) \geq \min\{\sdepth(U),\sdepth(N) \}.$$
\end{lema}

A similar result holds for Hilbert depth, see for instance \cite[Corollary 3.3]{uli}:

\begin{lema}\label{uli}(Hdepth Lemma)
If $0 \rightarrow U \rightarrow M \rightarrow N \rightarrow 0$ is a short exact sequence of finitely generated graded $S$-modules, then
$$\hdepth(M) \geq \min\{\hdepth(U),\hdepth(N) \}.$$
\end{lema}

We also recall the following well known results: See for instance \cite[Corollary 1.3]{asia}, \cite[Proposition 2.7]{mirci},
\cite[Theorem 1.1]{mir}, \cite[Lemma 3.6]{hvz} and \cite[Corollary 3.3]{asia}.

\begin{lema}\label{lem}
Let $I\subset S$ be a monomial ideal and let $u\in S$ a monomial which is not in $I$. We have that:
\begin{enumerate}
\item[(1)] $\sdepth(S/(I:u))\geq \sdepth(S/I)$.
\item[(2)] $\sdepth(I:u)\geq \sdepth(I)$.
\item[(3)] $\depth(S/(I:u))\geq \depth(S/I)$.
\item[(4)] If $I=u(I:u)$, then:
  \begin{enumerate}
  \item $\sdepth(S/(I:u))=\sdepth(S/I)$.
	\item $\depth(S/(I:u))=\depth(S/I)$.
	\item $\sdepth(I:u)=\sdepth(I)$.
  \end{enumerate}
\item[(5)] If $u$ is regular on $S/I$, then:
  \begin{enumerate}
	\item $\sdepth(S/(I,u))=\sdepth(S/I)-1$.
  \item $\depth(S/(I,u))=\depth(S/I)-1$.
  \end{enumerate}
\item[(6)] If $S'=S[x_{n+1}]$, then:
\begin{enumerate}
\item $\sdepth_{S'}(S'/IS')=\sdepth_S(S/I)+1$, $\sdepth_{S'}(IS')=\sdepth_S(I)+1$.
\item $\depth_{S'}(S'/IS')=\depth_S(S/I)+1$.
\end{enumerate}
\end{enumerate}
\end{lema}

\begin{lema}\label{lem7}
Let $I\subset S$ be a monomial ideal. Then the following assertions are equivalent:
\begin{enumerate}
\item[(1)] $\mathfrak m=(x_1,\ldots,x_n)\in \Ass(S/I)$.
\item[(2)] $\depth(S/I)=0$.
\item[(3)] $\sdepth(S/I)=0$.
\end{enumerate}
\end{lema}

Other useful results are the following:

\begin{teor}(See \cite[Theorem 2.4]{shen})\label{shen}

Let $I\subset S$ be a complete intersection monomial ideal with $|G(I)|=m$. Then 
$$\sdepth(I) = n-\left\lfloor \frac{m}{2} \right\rfloor.$$
\end{teor}

\begin{lema}(See \cite[Lemma 2.5]{okazaki})\label{oka}

Let $I\subset S$ be a monomial ideal with $G(I)=\{v_1,\ldots,v_m\}$. Assume that $x_n$ divides $v_i$ for $1\leq i\leq r$ but not for $r+1\leq i\leq m$,
where $1\leq r\leq m-1$. Let $J=(v_{r+1},\ldots,v_m)$. Then $I/x_n(I:x_n) \cong J\cap S'$, where $S'=K[x_1,\ldots,x_{n-1}]$.
\end{lema}

Lemma \ref{oka} plays a crucial role in the proof of the following result:

\begin{teor}(See \cite[Theorem 2.3]{okazaki})\label{okaz}
Let $I\subset S$ be a monomial ideal with $|G(I)|=m$. Then 
$$\sdepth(I)\geq n-\left\lfloor \frac{m}{2} \right\rfloor.$$
\end{teor}

Note that, according to Theorem \ref{shen} and Theorem \ref{okaz}, the complete intersection monomial case gives the minimal value for $\sdepth(I)$ in terms
of the number of minimal monomial generators of $I$.

In the following, we fix some notations and we recall the main result of \cite{lucrare2}. 
Let $0 \neq I\subset J\subset S$ be two square free monomial ideals. We consider the nonnegative integers
$$\alpha_k(J/I):=\# \{u\in S\;:\;u\text{ squarefree, with }u\in J\setminus I\text{ and }\deg(u)=k\},\;0\leq k\leq n.$$
For all $0\leq d\leq n$ and $0\leq k\leq d$, we consider the integers
\begin{equation}\label{betak}
  \beta_k^d(J/I):=\sum_{j=0}^k (-1)^{k-j} \binom{d-j}{k-j} \alpha_j(J/I).
\end{equation}
Note that, using an inverse formula, from \eqref{betak} we deduce that
\begin{equation}\label{alfak}
  \alpha_k(J/I):=\sum_{j=0}^k \binom{d-j}{k-j} \beta^d_j(J/I).
\end{equation}
With the above we have:

\begin{teor}(See \cite[Theorem 2.4]{lucrare2})\label{d1}

The Hilbert depth of $J/I$ is:
$$\hdepth(J/I):=\max\{d\;:\;\beta_k^d(J/I) \geq 0\text{ for all }0\leq k\leq d\}.$$
\end{teor}

We also recall the following basic results, regarding the Hilbert depth invariant:

\begin{prop}(See \cite[Proposition 2.8]{lucrare2})\label{p1}

We have that $$\hdepth(J/I)\geq \sdepth(J/I).$$
\end{prop}

\begin{prop}(See for instance \cite[Theorem 3.2]{lucrare3} and \cite[Theorem 3.4]{lucrare3})\label{p2}

We have that $$\dim(J/I)\geq \hdepth(J/I)\geq \depth(J/I).$$
\end{prop}

Note that, in the above proposition, we have equalities when $J/I$ is Cohen-Macaulay.

\newpage
\section{Main results}

\begin{dfn}
Suppose $n \geq 2$ and let $S=K[x_1,\ldots,x_n]$. The path graph $P_n$, of length $n-1$, is the graph on the
vertex set $V(P_n) = \{x_1,\ldots, x_n \}$ and the edge set $$E(P_n) = \{e_i=\{x_i, x_{i+1}\} :\;\text{ for }1 \leq i \leq  n-1\}.$$ 
The edge ideal of $P_n$ is $$I(P_n) = (x_1x_2, x_2x_3, \ldots , x_{n-1}x_n) \subset S.$$
\end{dfn}

For convenience, we denote $I_n=I(P_n)$. We recall the following results:

\begin{prop}\label{pro}
For any $n\geq 2$, we have that:
\begin{enumerate}
\item[(1)] $\depth(S/I_n)=\left\lceil \frac{n}{3} \right\rceil$. (\cite[Lemma 2.8]{mor})\vspace{5pt}
\item[(2)] $\sdepth(S/I_n)=\left\lceil \frac{n}{3} \right\rceil$. (\cite[Lemma 2.3]{stef})
\end{enumerate}
\end{prop}

\begin{cor}
For any $n\geq 2$ we have that $\hdepth(S/I_n)\geq \left\lceil \frac{n}{3} \right\rceil$.
\end{cor}

\begin{proof}
It follows from Proposition \ref{p1} and Proposition \ref{pro}(2)
    or Proposition \ref{p2} and Proposition \ref{pro}(1).		
\end{proof}

\begin{prop}\label{pin}
For any $n\geq 2$ we have that 
$$\left\lceil \frac{n+1}{2} \right\rceil \leq \sdepth(I_n) \leq \left\lceil \frac{2n+2}{3} \right\rceil.$$
In particular, $\hdepth(I_n)\geq \left\lceil \frac{n+1}{2} \right\rceil$.
\end{prop}

\begin{proof}
The first inequality follows from \cite[Corollary 1.5]{mir2}. In order to prove the second inequality,
we consider three cases: (i) $n=3k$, (ii) $n=3k+1$ and (iii) $n=3k+2$.
\begin{enumerate}
\item[(i)] $n=3k$. Let $u=x_{2}x_{5}\cdots x_{3k-1}$. Then $(I_n:u)=(x_1,x_3,x_4,x_6,\ldots,x_{3k-2},x_{3k})$ is a monomial prime ideal with $2k$ generators.
    From Lemma \ref{lem}(4) it follows that $$\sdepth(I_n:u)=n-k=2k= \left\lceil \frac{2n+2}{3} \right\rceil \geq \sdepth(I_n).$$
		
\item[(ii)] $n=3k+1$. Let $u=x_{2}x_{5}\cdots x_{3k-1}$. Then $(I_n:u)=(x_1,x_3,x_4,x_6,\ldots,x_{3k-2},x_{3k})$ is a monomial prime ideal with $2k$ generators.
     From Lemma \ref{lem}(4) it follows that $$\sdepth(I_n:u)=n-k=2k+1 = \left\lceil \frac{2n+2}{3} \right\rceil \geq \sdepth(I_n).$$
		
\item[(iii)] $n=3k+2$.	Let $u=x_{2}x_{5}\cdots x_{3k-1}$. $(I_n:u)=(x_1,x_3,x_4,x_6,\ldots,x_{3k-2},x_{3k}, x_{3k+1}x_{3k+2})$ 
      is a monomial prime ideal with $2k+1$ generators.
			From Lemma \ref{lem}(4) it follows that 
			$$\sdepth(I_n:u)=n-\left\lfloor \frac{2k+1}{2} \right\rfloor = n-k =2k+2= \left\lceil \frac{2n+2}{3} \right\rceil \geq \sdepth(I_n).$$
\end{enumerate}
Hence, the proof is complete.
\end{proof}

\pagebreak

\begin{prop}\label{pp}
For all $0\leq k\leq d$, we have that:
\begin{enumerate}
\item[(1)] $\beta_k^d(S/I_n)=\sum\limits_{j=0}^k (-1)^{k-j}\binom{d-j}{k-j}\binom{n-j+1}{j}$.
\item[(2)] $\beta_k^d(I_n)=\binom{n-d+k-1}{k} - \sum\limits_{j=0}^k (-1)^{k-j}\binom{d-j}{k-j}\binom{n-j+1}{j}$.
\item[(3)] $\beta_k^d(I_n)=\sum\limits_{\ell=1}^{\left\lfloor \frac{k}{2} \right\rfloor} (-1)^{\ell-1} \sum\limits_{j=0}^k (-1)^{k-j}\binom{d-j}{k-j}\binom{n-j+1}{\ell}\binom{n-2\ell}{j-2\ell}$.
\end{enumerate}
\end{prop}

\begin{proof}
(1) and (2) According to the case $m=2$ of \cite[Proposition 3.1]{lucrare4}, we have
that 
\begin{equation}\label{akin}
\alpha_k(S/I_n)=\binom{n-k+1}{k}\text{ and }\alpha_k(I_n)=\binom{n}{k}-\binom{n-k+1}{k}\text{ for all }0\leq k\leq n.
\end{equation}
In particular $\alpha_0(I_n)=\alpha_1(I_n)=0$. From \eqref{betak} and \eqref{akin} it follows that
\begin{equation}\label{bkin}
\begin{split}
\beta_k^d(I_n) = \sum_{j=0}^k (-1)^{k-j}\binom{d-j}{k-j}\binom{n-j+1}{j}\text{ and }\\
\beta_k^d(I_n) = \sum_{j=0}^k (-1)^{k-j}\binom{d-j}{k-j}\left( \binom{n}{j} - \binom{n-j+1}{j} \right)\text{ for all }0\leq j\leq d\leq n.
\end{split}
\end{equation}
The conclusion follows from \eqref{bkin} and the identity
\begin{equation}\label{minune}
\binom{n-d+k-1}{k} = \sum_{j=0}^k (-1)^{k-j}\binom{d-j}{k-j}\binom{n}{j}.
\end{equation}
(3) Using (2) and the identity
$$\binom{n-j+1}{j} = \sum_{\ell=0}^{\left\lfloor \frac{j}{2} \right\rfloor} (-1)^{\ell}\binom{n-j+1}{\ell}\binom{n-2\ell}{j-2\ell},$$
we deduce that
$$\beta_k^d(I_n) = \binom{n-d+k-1}{k} - \sum_{j=0}^k (-1)^{k-j}\binom{d-j}{k-j} \sum_{\ell=0}^{\left\lfloor \frac{j}{2} \right\rfloor} (-1)^{\ell}\binom{n-j+1}{\ell}\binom{n-2\ell}{j-2\ell}.$$
By changing the order of summation, we get
$$\beta_k^d(I_n) = \binom{n-d+k-1}{k} - \sum_{\ell=0}^{\left\lfloor \frac{k}{2} \right\rfloor} (-1)^{\ell} \sum_{j=0}^k (-1)^{k-j}\binom{d-j}{k-j}\binom{n-j+1}{\ell}\binom{n-2\ell}{j-2\ell}.$$
Taking $\ell=0$ we note that 
$$(-1)^{0}\sum_{j=0}^k (-1)^{k-j}\binom{d-j}{k-j}\binom{n-j+1}{0}\binom{n-2\cdot 0}{j-2\cdot 0}=\binom{n-d+k-1}{k}.$$ 
Hence, we get the required formula.
\end{proof}

\begin{teor}\label{cory}
We have that:
\begin{enumerate}
\item[(1)] $\hdepth(S/I_n)=\max\{d\;:\;\sum\limits_{j=0}^k (-1)^{k-j}\binom{d-j}{k-j}\binom{n-j+1}{j} \geq 0 \text{ for all }1\leq k\leq d\}$.
\item[(2)] $\hdepth(I_n)=\max\{d\;:\;\sum\limits_{j=0}^k (-1)^{k-j}\binom{d-j}{k-j}\binom{n-j+1}{j} \leq \binom{n-d+k-1}{k} \text{ for all }2\leq k\leq d\}$.
\end{enumerate}
\end{teor}

\begin{proof}
(1) Since $\alpha_0(S/I_n)=1$ and $\alpha_1(S/I_n)=n$ it follows that $\beta_0^d(S/I_n)=1$ and $\beta_1^d(S/I_n)=n-d\geq 0$.
Hence, the result follows from Proposition \ref{pp}(1) and Thereom \ref{d1}.

(2) Since $\alpha_0(I_n)=\alpha_1(I_n)=0$, and thus $\beta_0(I_n)=\beta_1(I_n)=0$, $\beta_2(I_n)=\alpha_2(I_n)=n-1$, 
the result follows from Proposition \ref{pp}(2) and Thereom \ref{d1}.
\end{proof}

Based on our computer experiments, we propose the following Conjecture:

\begin{conj}\label{conj}
For all $n\geq 2$, we have that $$\hdepth(I_n)\geq \left\lfloor \frac{2n+1}{3} \right\rfloor.$$
\end{conj}

As a direct consequence of Proposition \ref{cory}, we have the following:

\begin{cor}
For $n\geq 2$, the following are equivalent:
\begin{enumerate}
\item[(1)] Conjecture \ref{conj} holds for $n$.
\item[(2)] $\sum\limits_{j=0}^k (-1)^{k-j}\binom{\left\lfloor \frac{2n+1}{3} \right\rfloor-j}{k-j}\binom{n-j+1}{j} \leq \binom{\left\lfloor \frac{n+1}{3} \right\rfloor+k-1}{k}$ for all $2\leq k\leq \left\lfloor \frac{2n+1}{3} \right\rfloor$.
\end{enumerate}
\end{cor}

\begin{obs}\label{obsy}\rm
Using a C\texttt{++} code, we verify Conjecture \ref{conj} for $n\leq 2000$.

Moreover, we noted that, 
$\hdepth(S/I_n)=\left\lceil \frac{n}{3} \right\rceil$ for $2\leq n\leq 9$, $\hdepth(S/I_{10})=4$ and
$$\hdepth(S/I_n) \geq  \left\lceil \frac{n}{3} \right\rceil + \left\lfloor \frac{3n-1}{29} \right\rfloor - 1\text{ for }11\leq n\leq 1000.$$
Moreover, the bound seems to be sharp, that is $\hdepth(S/I_n)-\left\lceil \frac{n}{3} \right\rceil - \left\lfloor \frac{3n-1}{29} \right\rfloor +1 \in \{0,1\}$
for $10\leq n\leq 521$ and $\hdepth(S/I_n)-\left\lceil \frac{n}{3} \right\rceil - \left\lfloor \frac{3n-1}{29} \right\rfloor +1 \in \{0,1,2\}$ for
$522 \leq n \leq 1000$.

We noted also that
$$\hdepth(I_n)\geq \left\lfloor \frac{2n+1}{3} \right\rfloor + \left\lfloor \frac{2n-5}{17} \right\rfloor - 1\text{ for }11\leq n\leq 1000,$$
and this bound also seems to be sharp, as 
$\hdepth(I_n) - \left\lfloor \frac{2n+1}{3} \right\rfloor - \left\lfloor \frac{2n-5}{17} \right\rfloor + 1\in \{0,1\}$ for.
\end{obs}

\begin{dfn}
Suppose $n \geq 3$ and let $S=K[x_1,\ldots,x_n]$. The cycle graph $C_n$, of length $n$, is the graph on the
vertex set $V(C_n) = \{x_1,\ldots, x_n \}$ and the edge set $$E(C_n) = \{e_i=\{x_i, x_{i+1}\} :\;\text{ for }1 \leq i \leq  n-1\}\cup\{e_n=\{x_n,x_i\}\}.$$ 
The edge ideal of $C_n$ is $I(C_n) = (x_1x_2, x_2x_3, \ldots , x_{n-1}x_n,x_nx_1) \subset S.$
\end{dfn}

For convenience, we denote $J_n=I(C_n)$. We recall the following results: \pagebreak

\begin{prop}\label{pro2}
For any $n\geq 3$, we have that:
\begin{enumerate}
\item[(1)] $\depth(S/J_n)=\left\lceil \frac{n-1}{3} \right\rceil$. (\cite[Proposition 1.3]{mir2})
\item[(2)] $\sdepth(S/J_n)=\left\lceil \frac{n-1}{3} \right\rceil$ for $n\equiv 0(\bmod\;3)$ and $n\equiv 2(\bmod\;3)$. (\cite[Theorem 1.9]{mir2})
\item[(3)] $\left\lceil \frac{n-1}{3} \right\rceil \leq \sdepth(S/J_n)\leq \left\lceil \frac{n}{3} \right\rceil$ for $n\equiv 1(\bmod\;3)$. (\cite[Theorem 1.9]{mir2})
\end{enumerate}
\end{prop}

\begin{cor}\label{corcor}
For any $n\geq 3$ we have that $\hdepth(S/J_n)\geq \left\lceil \frac{n-1}{3} \right\rceil$.
\end{cor}

\begin{proof}
It follows from Proposition \ref{pro2}(1) and Proposition \ref{p2}.
\end{proof}

\begin{prop}\label{pjn}
For any $n\geq 3$, we have that:
$$ \left\lceil \frac{n}{2} \right\rceil \leq \sdepth(J_n) \leq \left\lceil \frac{2n+2}{3} \right\rceil.$$
Also, $\hdepth(J_n)\geq \left\lceil \frac{n}{2} \right\rceil$.
\end{prop}

\begin{proof}
The first inequality follows from \cite[Corollary 1.5]{mir2}. The proof of the second inequality is similar to
the proof of Proposition \ref{pin}, so we omit it.

The last assertion follows from Proposition \ref{p1}.
\end{proof}

\begin{prop}\label{pp2}
For all $0\leq k\leq d$, we have that:
\begin{enumerate}
\item[(1)] $\beta_k^d(S/J_n)=\sum\limits_{j=0}^k (-1)^{k-j}\binom{d-j}{k-j}\binom{n-j}{j} + \sum\limits_{j=1}^{k} (-1)^{k-j}\binom{d-j}{k-j}\binom{n-j-1}{j-1}$.

\item[(2)] $\beta_k^d(J_n)=\binom{n-d+k-1}{k}-\sum\limits_{j=0}^k (-1)^{k-j}\binom{d-j}{k-j}\binom{n-j}{j}-
             \sum\limits_{j=1}^k (-1)^{k-j}\binom{d-j}{k-j}\binom{n-j-1}{j-1}$.
\end{enumerate}
\end{prop}

\begin{proof}
According to \cite[Proposition 4.3]{lucrare4}, we have that
$$\alpha_j(S/J_n)=\binom{n-j+1}{j}-\binom{n-j-1}{j-2} = \binom{n-j}{j}+\binom{n-j-1}{j-1},\;0\leq j\leq n.$$
Hence, from \eqref{betak} and \eqref{minune} it follows that
$$\beta_k(S/J_n)=\sum\limits_{j=0}^k (-1)^{k-j}\binom{d-j}{k-j}\binom{n-j}{j}+\sum\limits_{j=1}^k (-1)^{k-j}\binom{d-j}{k-j}\binom{n-j-1}{j-1}.$$
On the other hand, we have that
$$\sum\limits_{j=1}^k (-1)^{k-j}\binom{d-j}{k-j}\binom{n-j-1}{j-1} = - \sum_{j=0}^{k-1} (-1)^{k-j} \binom{d-1-j}{k-1-j}\binom{n-j-2}{j} $$
Thus, we proved (1). As $\alpha_j(J_n)=\binom{n}{j}-\alpha_j(S/J_n)$ and thus $$\beta_k^d(J_n)=\binom{n-d+k-1}{k}-\beta_k^d(S/J_n),$$ (2) follows from (1).
\end{proof}

\pagebreak

\begin{teor}\label{cory2}
We have that:\small
\begin{enumerate}
\item[(1)] $\hdepth(S/J_n)=\max\{d:\;\sum\limits_{j=0}^k (-1)^{k-j}\binom{d-j}{k-j}\left(\binom{n-j}{j}+\binom{n-j-1}{j-1}\right) \geq 0,\text{ for all }
1\leq k\leq d\}$.
\item[(2)] $\hdepth(J_n)=\max\{d:\;\sum\limits_{j=0}^k (-1)^{k-j}\binom{d-j}{k-j}\left(\binom{n-j}{j}+\binom{n-j-1}{j-1}\right)\leq \binom{n-d+k-1}{k},\; 2\leq k\leq d\}.$ \normalsize
\end{enumerate}
\end{teor}

\begin{proof}
(1) Since $\alpha_0(S/J_n)=1$ and $\alpha_1(S/J_n)=n$ it follows that $\beta_0^d(S/J_n)=1$ and $\beta_1^d(S/J_n)=n-d$.
Hence, the result follows from Proposition \ref{pp2}(1) and Thereom \ref{d1}.

(2) Since $\alpha_0(J_n)=\alpha_1(J_n)=0$, and thus $\beta_0(J_n)=\beta_1(J_n)=0$, $\beta_2(J_n)=\alpha_2(J_n)=n$, 
the result follows from Proposition \ref{pp2}(2) and Thereom \ref{d1}.
\end{proof}

Based on our computer experiments, we conjecture that:

\begin{conj}\label{conj2}
For all $n\geq 3$, it holds that:
\begin{enumerate}
\item[(1)] $\hdepth(S/I_n)-\hdepth(S/J_n) \in \{0,1\}$.
\item[(2)] $\hdepth(I_n)-\hdepth(J_n) \in \{0,1\}$.
\item[(3)] $\hdepth(J_n)\geq \left\lfloor \frac{2n}{3} \right\rfloor$.
\end{enumerate}
\end{conj}

We were able to verify Conjecture \ref{conj2} for $n\leq 1000$. Moreover, we also propose the following:

\begin{conj}\label{conj3}
We have that:
\begin{align*}
& (1)\; \lim\limits_{N\to\infty} \frac{\# \{n\leq N\;:\;\hdepth(S/I_n) = \hdepth(S/J_n)\}}{N} = \frac{2}{3},\\
& (2)\; \lim\limits_{N\to\infty} \frac{\# \{n\leq N\;:\;\hdepth(I_n) = \hdepth(J_n)\}}{N} = \frac{5}{6}.
\end{align*}
\end{conj}

\begin{teor}\label{cory3}
For all $n\geq 6$ we have that 
$$\hdepth(J_n/I_n)=2+\hdepth(K[x_1,\ldots,x_{n-4}]/I_{n-4})\geq \left\lceil \frac{n+2}{3} \right\rceil.$$
\end{teor}

\begin{proof}
According to \cite[Proposition 4.2]{lucrare4}, we have that $\alpha_j(J_n/I_n) = \binom{n-j-1}{j-2}$. Therefore,
it follows that $\alpha_0(J_n/I_n)=\alpha_1(J_n/I_n)=0=\alpha_{n-1}(J_n/I_n)=\alpha_{n}(J_n/I_n)$ and 
$$\alpha_j(J_n/I_n)=\binom{(n-4)-(j-2)+1}{j-2} = \alpha_{j-2}(K[x_1,\ldots,x_{n-4}]/I_{n-4})\text{ for }2\leq j\leq n-2.$$
From \eqref{betak} it follows that $\beta_0^d(J_n/I_n)=\beta_1^d(J_n/I_n)=0$ and 
$$\beta_k^d(J_n/I_n)=\beta_{k-2}^{d-2}(K[x_1,\ldots,x_{n-4}]/I_{n-4}).$$
Hence, the required conclusion follows from Theorem \ref{d1} and Corollary \ref{corcor}.
\end{proof}

Note that, according to \cite[Proposition 1.10]{mir2}, we have 
$$\sdepth(J_n/I_n)=\depth(J_n/I_n)=\left\lceil \frac{n+2}{3} \right\rceil.$$
Hence, the last inequality in Theorem \ref{cory3} follows also from the above and Proposition \ref{p1} or Proposition \ref{p2}.

\newpage
\section{Applications}

\subsection*{Star graphs}

\begin{dfn}
Let $S=K[x_1,\ldots,x_n]$ and $\overline S=S[y]$.
The \emph{star graph} $\St_n$ is the graph on the vertex set $V(\St_n)=\{x_1,\ldots,x_n,y\}$, with the edge set $E(\St_n)=\{\{x_i,y\}\;:\;1\leq i\leq n\}$.
The edge ideal of $\St_n$ is $$I(\St_n)=(x_1y,\ldots,x_ny)\subset \overline S.$$
\end{dfn}

\begin{obs}\rm
According to \cite[Theorem 2.9]{lucrare5}, $\hdepth(I(\St_n))=\left\lfloor \frac{n+3}{2} \right\rfloor$.
On the other hand, it seems to be very difficult to compute $\hdepth(\overline S/I(\St_n))$.
\end{obs}

\begin{prop}
Let $I\subset S$ be a monomial ideal with $\depth(S/I)\geq 1$ and let $L=I+I(\St_n)\subset \overline S$. Then:
$$\dim(\oS/L)\geq \hdepth(\oS/L)\geq \sdepth(\oS/L)\geq \depth(\oS/L)=1.$$
\end{prop}

\begin{proof}
Note that, since $\depth(S/I)\geq 1$, from Lemma \ref{lem7} it follows that $\sdepth(S/I)\geq 1$.
Also, as $(L:y)=\me\oS$ and $(L,y)=(I,y)$, we have the short exact sequence
\begin{equation}
0\to \oS/\me\oS \cong K[y] \to \oS/L \to \oS/(L,y)\cong S/I \to 0.
\end{equation}
From Lemma \ref{lem1}, it follows that $\depth(\oS/L)=1$. Also, from Lemma \ref{asia}, it follows that 
$\sdepth(\oS/L)\geq 1$. Hence, by applying Proposition \ref{p2}, we get the required result.
\end{proof}

\subsection*{Generalized star graphs}


\begin{dfn}(See \cite[Definition 2.1]{ali})\label{gsg}
Assume that $k,n_1,\ldots,n_k$ are positive integers. Let $S:=K[y,x_{j,i}\;:\;1\leq j\leq n_i,\;1\leq i\leq k]$.
For each $1\leq i\leq k$ we consider the ideal 
$$I_i=(yx_{1,i},\; x_{1,i}x_{2,i},\; \ldots,x_{n_i-1,i}x_{n_i,i})\subset S_i:=K[y,x_{1,i},x_{2,i},\ldots,x_{n_i,i}].$$
Let $I=I_1S + I_2S +\cdots + I_kS \subset S$. A graph $G$ with the vertex sex
$$V(G) = \{y,\;x_{j,i}\;:\;1\leq j\leq n_i,\;1\leq i\leq k\},$$
and $I=I(G)$ is called a $(k;n_1,\ldots,n_k)$-star graph.
\end{dfn}

Note that, for each $1\leq i\leq k$, the ideal $I_i$ is the edge ideal of a path of length $n_i$.

We recall the following result:

\begin{teor}(See \cite[Theorem 2.6]{ali} and \cite[Theorem 2.7]{ali})\label{teo}

With the above notations, we have that:
\begin{enumerate}
\item[(1)] If $n_i\equiv 0\text{ or }2(\bmod\;3)$ for all $1\leq i\leq k$ then 
           $$\sum\limits_{i=1}^{k} \left\lceil \frac{n_i}{3} \right\rceil \leq \depth(S/I),\sdepth(S/I) \leq 
           \sum\limits_{i=1}^{k} \left\lceil \frac{n_i}{3} \right\rceil +1.$$
\item[(2)] If $n_i\equiv 1(\bmod\;3)$ for some $1\leq i\leq k$ then 
           $$\sdepth(S/I)=\depth(S/I)=1+\sum\limits_{i=1}^{k}\left\lceil \frac{n_i-1}{3} \right\rceil.$$
\end{enumerate}
\end{teor}

From Theorem \ref{teo} and Proposition \ref{p1} or Proposition \ref{p2} we get the following result:

\begin{cor}
With the above notations, we have that:
\begin{enumerate}
\item[(1)] If $n_i\equiv 0\text{ or }2(\bmod\;3)$ for all $1\leq i\leq k$ then $\hdepth(S/I)\geq \sum\limits_{i=1}^{k} \left\lceil \frac{n_i}{3} \right\rceil$.
\item[(2)] If $n_i\equiv 1(\bmod\;3)$  for some $1\leq i\leq k$ then $\hdepth(S/I)\geq 1+\sum\limits_{i=1}^{k}\left\lceil \frac{n_i-1}{3} \right\rceil$.
\end{enumerate}
\end{cor}

If $u\in K[x_1,\ldots,x_n]$ is a monomial, the support of $u$ is $\supp(u)=\{x_i\;:\;x_i\mid u\}$.
We also denote $N=\sum_{i=1}^k n_i$. 

\begin{teor}\label{t4}
With the above notation, we have that 
$$1 + \left\lfloor \frac{\varepsilon}{2} \right\rfloor + \sum_{i=1}^k \left\lfloor \frac{2n_i}{3} \right\rfloor  \geq \sdepth(I)\geq \left\lceil \frac{N+k}{2} \right\rceil,$$
where $\varepsilon=\# \{i\;:\;n_i\equiv\;0,2(\bmod\;3)\}$. In particular, $\hdepth(I)\geq \left\lceil \frac{N+k}{2} \right\rceil$.
\end{teor}

\begin{proof}
We consider the short exact sequence
\begin{equation}\label{ek1}
0 \to y(I:y) \to I \to \frac{I}{y(I:y)} \to 0.
\end{equation}
We consider the ideals:
$$L_i=(x_{i,1}x_{i,2},\; x_{i,2}x_{i,3},\; \ldots,x_{i,n_i-1}x_{i,n_i}),\;U_i=(x_{i_1},\; x_{i,2}x_{i,3},\; \ldots,x_{i,n_i-1}x_{i,n_i})\text{ for }1\leq i\leq k.$$
Note that 
\begin{equation}\label{ek2}
(I:y)=\sum\limits_{i=1}^k U_i.
\end{equation} 
On the other hand, from Lemma \ref{oka} it follow that 
\begin{equation}\label{ek3}
\frac{I}{y(I:y)} \cong \sum\limits_{i=1}^k L_i\cap S',\text{ where }S'=K[x_{j,i}\;:\;1\leq j\leq n_i,\;1\leq i\leq k].
\end{equation}
From \eqref{ek1}, \eqref{ek2}, \eqref{ek3} and Theorem \ref{okaz} it follows that
$$\sdepth(I)\geq \min\{\sdepth(I:y),\sdepth(I/y(I:y))\} \geq \left\lceil \frac{N+k}{2} \right\rceil.$$
As in the proof of Proposition \ref{pin}, we let 
$$\ell_i:=\left\lfloor \frac{n_i-1}{3} \right\rfloor\text{ and }u_i=x_{i,3}x_{i,6}\cdots x_{i,3\ell_i}\in S\text{ for }1\leq i\leq k,$$
and we note that $(U_i:u_i)$ is minimally generated by $1+2\ell_i$ generators, 
if $n_i\equiv 0,2(\bmod\;3)$ or $2+2\ell_i$ if $n_i\equiv 1(\bmod\;3)$ generators.
Moreover, the monomials from $G(U_i:u_i)$ have disjoint supports.
Now, the conclusion follows from \eqref{ek2}, Lemma \ref{lem}(2), Theorem \ref{shen} and the fact that 
$n_i - 1 - \ell_i = \left\lfloor \frac{2n_i}{3} \right\rfloor,\text{ for all }1\leq i\leq k$.
\end{proof}

\subsection*{Double broom graphs}


\begin{dfn}(See \cite[Page 382]{anda})\label{dbg}
Let $n_1,n_2,n\geq 2$ be some integers. In the ring of polynomials $S:=K[x_1,\ldots,x_{n_1},y_1,\ldots,y_n,z_1,\ldots,z_{n_2}]$,
we consider the ideals:
\begin{align*}
& I_1=(x_1y_1,x_2y_1,\ldots,x_{n_1}y_1)\subset S_1=K[y_1,x_1,\ldots,x_{n_1}], \\
& I_2=(y_1y_2,y_2y_3,\ldots,y_{n-1}y_n)\subset S_2=K[y_1,y_2,\ldots,y_n], \\
& I_3=(y_nz_1,y_nz_2,\ldots,y_nz_{n_2})\subset S_3=K[y_n,z_1,\ldots,z_{n_2}].
\end{align*}
Also, we let $I=I_1S+I_2S+I_3S$. The double broom $P(n_1,n,n_2)$ is a graph on the vertex set
$$V=\{x_1,\ldots,x_{n_1},y_1,\ldots,y_n,z_1,\ldots,z_{n_2}\},$$
such that $I(P(n_1,n,n_2))=I$.

In particular, if $n=2$ then $P(n_1,2,n_2)$ is called a \emph{double star}.
\end{dfn}

\begin{teor}\label{teor2}
With the above notations, we have that $$\hdepth(S/I)\geq \sdepth(S/I)=\depth(S/I)=2+\left\lceil \frac{n-2}{3} \right\rceil.$$
\end{teor}

\begin{proof}
It is easy to see that 
\begin{equation}\label{pisi}
(I:y_1y_n)=(x_1,x_2,\ldots,x_{n_1},z_1,z_2,\ldots,z_{n_2})+(y_2y_3,y_3y_2,\cdots,y_{n-2}y_{n-1}).
\end{equation}
From \eqref{pisi} it follows that
\begin{equation}\label{cat}
\frac{S}{(I:y_1y_n)}\cong \frac{S_2}{(y_2y_3,y_3y_2,\cdots,y_{n-2}y_{n-1})} \cong 
\frac{K[y_2,\ldots,y_{n-1}]}{(y_2y_3,y_3y_2,\cdots,y_{n-2}y_{n-1})}[y_1,y_n].
\end{equation}
From Lemma \ref{lem}, Proposition \ref{pro} and \eqref{cat} it follows that
$$\depth(S/I),\sdepth(S/I)\leq 2+\left\lceil \frac{n-2}{3} \right\rceil.$$
In order to prove the other inequalities, we
let $$L=(y_nz_1,\ldots,y_nz_{n_2},y_2y_3,y_3y_2,\cdots,y_{n-1}y_{n})\subset S':=K[y_1,\ldots,y_n,z_1,\ldots,z_{n_2}].$$
We note that
$(I:y_1) = (x_1,x_2,\ldots,x_{n_1})+LS\text{ and }(I,y_1)=(y_1)+LS.$
In particular, we have 
$$\frac{S}{(I:y_1)}\cong \frac{S'}{L} \text{ and }\frac{S}{(I,y_1)}\cong \frac{S'}{(y_1,L)}[x_1,\ldots,x_{n_1}].$$
Therefore, from Lemma \ref{lem} it follows that 
\begin{equation}\label{sica}
\begin{split}
&\depth(S/(I,y_1)) = \depth(S/(I:y_1))-n_1+1\text{ and }\\
& \sdepth(S/(I,y_1))=\depth(S/(I:y_1))-n_1+1.
\end{split}
\end{equation}
From the short exact sequence $0\to S/(I:y_1) \to S/I \to S/(I,y_1) \to 0$, \eqref{sica}, Lemma \ref{lem}, Lemma \ref{lem1} and
Lemma \ref{asia} it follows that
\begin{equation}
\begin{split}
& \depth(S/I)=\depth(S/(I:y_1))=\depth(S'/L)\text{ and }\\
& \sdepth(S/I)=\sdepth(S/(I:y_1))=\sdepth(S'/L).
\end{split}
\end{equation}
Similarly to \eqref{pisi} we note that 
$$(L:y_n)=(z_1,z_2,\ldots,z_{n_2})+(y_2y_3,y_3y_2,\cdots,y_{n-2}y_{n-1})\subset S'$$
Also, we note that $$(L,y_n)=(y_n)+(y_2y_3,y_3y_2,\cdots,y_{n-2}y_{n-1})\subset S'$$
Using Lemma \ref{lem}, Lemma \ref{lem1} and Lemma \ref{asia} it follows that
\begin{equation}\label{ciuciu}
\depth(S'/L)\geq \depth(S'/(L:y_n))\text{ and }\sdepth(S'/L)\geq \sdepth(S'/(L:y_n)).
\end{equation}
Since $(I:y_1y_n)=(L:y_n)S$, the conclusion follows from Lemma \ref{lem} and \eqref{ciuciu}.
\end{proof}

\begin{prop}\label{p10}
We have that:
\begin{enumerate}
\item[(1)] $\sdepth(I)\leq \left\lceil \frac{n_1+n_2}{2} \right\rceil + \left\lceil \frac{2n+1}{3} \right\rceil + 1$, if  $n \equiv 0,2(\bmod\;3)$.
\item[(2)] $\sdepth(I)\leq \left\lceil \frac{n_1+n_2}{2} \right\rceil + \left\lceil \frac{2n+1}{3} \right\rceil $, if  $n \equiv 1(\bmod\;3)$.
\item[(3)] $\hdepth(I)\geq \sdepth(I)\geq \left\lceil \frac{n_1+n_2+n+1}{2} \right\rceil$.
\end{enumerate}
\end{prop}

\begin{proof}
Let $k=\left\lfloor \frac{n-2}{3} \right\rfloor$. From \eqref{pisi} and similarly to the proof of Proposition \ref{pin} it follows that 
\begin{align*}
& (I:y_1y_ny_3y_6\cdots y_{3k}) = (x_1,x_2,\ldots,x_{n_1},z_1,z_2,\ldots,z_{n_2})+U,\text{ where }\\
& U=(y_2,y_4,y_5,y_7,\ldots,y_{3k-1},y_{3k+1})\text{ if }n=3k+2,\\
& U=(y_2,y_4,y_5,y_7,\ldots,y_{3k-1},y_{3k+1})\text{ if }n=3k+3,\\
& U=(y_2,y_4,y_5,y_7,\ldots,y_{3k-1},y_{3k+1},y_{3k+2}y_{3k+3})\text{ if }n=3k+4.
\end{align*}
Note that $ (I:y_1y_ny_3y_6\cdots y_{3k})$ is a monomial complete intersection generated by $n_1+n_2+2k$ monomials
if $n \equiv 0,2(\bmod\;3)$ or $n_1+n_2+2k+1$ if $n \equiv 1(\bmod\;3)$. Hence, we proved (1) and (2).

In order to prove (3), it is enough to note that $I$ is minimally generated by $n_1+n_2+n-1$ monomial and to apply Theorem \ref{okaz}.
\end{proof}

\subsection*{Double star graph}

\noindent\\
\vspace{10pt}
In the following, $I$ is the edge ideal of the double star graph $P(n_1,2,n_2)$. Let $N:=n_1+n_2+2$.

\begin{lema}\label{starl}
We have that $\alpha_0(I)=\alpha_1(I)=0$ and 
$$\alpha_j(I)=\binom{n_1}{j-1}+\binom{n_2}{j-1}+\binom{n_1+n_2}{j-2},\text{ for all }j\geq 2.$$
\end{lema}

\begin{proof}
Let $u\in I$ be a squarefree monomial of degree $j\leq 2$. We have three (disjoint) cases:
\begin{itemize}
\item $y_1\mid u$ and $y_2\nmid u$. Then $u=y_1w$ where $w\in K[x_1,\ldots,x_{n_1}]$ is a squarefree monomial of degre $j-1$.
\item $y_2\mid u$ and $y_1\nmid u$. Then $u=y_2w$ where $w\in K[z_1,\ldots,z_{n_2}]$ is a squarefree monomial of degre $j-1$.
\item $y_1y_2\mid u$. Then $u=y_1y_2w$ where $w\in K[x_1,\ldots,x_{n_1},z_1,\ldots,z_{n_2}]$ is a squarefree monomial of degre $j-2$.
\end{itemize}
The conclusion follows easily.
\end{proof}

\begin{prop}\label{p12}
We have that:
\begin{enumerate}
\item[(1)] For all $0\leq d\leq N$ we have $\beta_0^d(I)=\beta_1^d(I)=0$ and for all $2\leq k\leq d\leq N$ we have
           $$\beta_k^d(I)=\binom{n_1-d+k-1}{k-1}+\binom{n_2-d+k-1}{k-1}+\binom{n_1+n_2-d+k-1}{k-2} + 2\cdot(-1)^k\binom{d-1}{k-1}.$$
\item[(2)] For all $0\leq d\leq N$ we have $\beta_0^d(S/I)=1$, $\beta_1^d(S/I)=N-d$. Also, we have
           \begin{align*} 
					 \beta_k^d(S/I) &= \binom{n_1+n_2-d+k+1}{k}-\binom{n_1+n_2-d+k-1}{k-2}-\binom{n_1-d+k-1}{k-1}\\
					              &-\binom{n_2-d+k-1}{k-1}+ 2\cdot(-1)^{k-1}\binom{d-1}{k-1},\text{ for all }2\leq k\leq d\leq N.
           \end{align*}												
\end{enumerate}
\end{prop}

\begin{proof}
(1) Since $\alpha_0(I)=\alpha_1(I)=0$ it follows that $\beta_0^d(I)=\beta_1^d(I)=0$. Assume $d\geq 2$. From \eqref{betak} and Lemma \ref{starl} it follows that
\begin{equation}\label{ee1}
\beta_k^d(I) = \sum_{j=2}^{k} (-1)^{k-j} \binom{d-j}{k-j} \left( \binom{n_1}{j-1} + \binom{n_2}{j-1} + \binom{n_1+n_2}{j-2}\right).
\end{equation}
We have
$$\sum_{j=2}^{k} (-1)^{k-j} \binom{d-j}{k-j} \binom{n_1}{j-1} = \sum_{j=1}^{k-1} (-1)^{(k-1)-j} \binom{(d-1)-j}{(k-1)-j} \binom{n_1}{j}.$$
Hence, from \eqref{minune} it follows that 
\begin{equation}\label{ee2}
\sum_{j=2}^{k} (-1)^{k-j} \binom{d-j}{k-j} \binom{n_1}{j-1} = \binom{n_1-d+k-1}{k-1} + (-1)^k \binom{d-1}{k-1}.
\end{equation}
Similarly, we have
\begin{equation}\label{ee3}
\sum_{j=2}^{k} (-1)^{k-j} \binom{d-j}{k-j} \binom{n_2}{j-1} = \binom{n_2-d+k-1}{k-1} + (-1)^k \binom{d-1}{k-1}.
\end{equation}
On the other hand, we have
\begin{equation}\label{ee4}
\sum_{j=2}^{k} (-1)^{k-j} \binom{d-j}{k-j} \binom{n_1+n_2}{j-2} = \sum_{j=0}^{k-2} (-1)^{k-2-j} \binom{d-2-j}{k-2-j}\binom{n_1+n_2}{j}.
\end{equation}
The conclusion follows from \eqref{ee1}, \eqref{ee2}, \eqref{ee3}, \eqref{ee4} and \eqref{minune}.
\end{proof}

\begin{obs}\rm
Proposition \ref{p12} and Theorem \ref{d1} gives formulas for $\hdepth(I)$ and $\hdepth(S/I)$. 
From Theorem \ref{teor2} of Proposition \ref{p12} we note that $\hdepth(S/I)\geq 2$.
Also, from Proposition \ref{p10} we deduce that $\hdepth(I)\geq \left\lceil \frac{n_1+n_2+3}{2} \right\rceil$.
\end{obs}

\subsection*{Data availability}

Data sharing not applicable to this article as no data sets were generated or analyzed during the current study.

\subsection*{Conflict of interest}

The authors have no relevant financial or non-financial interests to disclose.

\end{document}